\documentclass[a4paper]{article}

\usepackage[latin1]{inputenc} 
\usepackage{graphicx}
\usepackage[pdftex,bookmarks]{hyperref}
\usepackage{amsmath}    
\usepackage{amssymb}
\usepackage{amsmath,amsthm}
\usepackage{textcomp}
\usepackage[all]{xy}
\usepackage{geometry}

\newtheorem{teo}{Theorem}[section]
\newtheorem{defi}{Definition}[section]

\newtheorem{cor}{Corollary}[section]
\newtheorem{prop}{Proposition}[section]

\DeclareMathOperator{\im}{im}

\DeclareMathOperator{\dvol}{dvol}

\DeclareMathOperator{\ric}{Ric}

\DeclareMathOperator{\vol}{vol}

\title{\huge \bf Von Neumann dimension, Hodge index theorem and geometric applications}
\author{Francesco Bei  \bigskip \\
Institut Camille Jordan, Universit\'e de Lyon1,\\ E-mail addresses: \ bei@math.univ-lyon1.fr \     francescobei27@gmail.com }
\date{}

\begin{document}

\maketitle

\begin{abstract}
This note contains a reformulation of the Hodge index theorem within the framework of Atiyah's $L^2$-index theory. More precisely, given a compact K\"ahler manifold $(M,h)$ of even complex dimension $2m$, we prove that $$\sigma(M)=\sum_{p,q=0}^{2m}(-1)^ph_{(2),\Gamma}^{p,q}(M)$$ where $\sigma(M)$ is the signature of $M$ and  $h_{(2),\Gamma}^{p,q}(M)$ are the $L^2$-Hodge numbers of $M$ with respect to a Galois covering  having $\Gamma$ as group of Deck transformations. Likewise  we also prove an $L^2$-version of the Fr\"olicher index theorem, see \eqref{adriano}. Afterwards we give some applications of these two theorems and finally we conclude this paper by collecting other properties of the $L^2$-Hodge numbers. 
\end{abstract}

\vspace{0.3 cm}


\tableofcontents

\section*{Introduction}
The purpose of this paper is to collect some applications of the $L^2$-index theory in the setting of compact complex manifolds. $L^2$-index theory was introduced in 1976 by Atiyah in his seminal paper \cite{MA}. Since then it became a deep and rich research topic with several applications to the geometry and topology of manifolds. Just to mention some of its great achievements we can recall the construction of new families of topological invariants such as the $L^2$-Betti numbers and the Novikov-Shubin invariants, the construction of more refined invariants such as the $L^2$-analytic torsion and  the Cheeger-Gromov $\rho^{(2)}$-invariant and many other applications to the geometry of manifolds such as the solution provided by Gromov, in the framework of negatively curved K\"ahler manifolds, of the Hopf conjecture. We invite the reader to consult  \cite{Luck} for an in-depth discussion on this topic. In this paper we focus on two classical results of complex geometry, the Hodge index theorem and a theorem due to Fr\"olicher which gives the Euler characteristic of a compact complex manifold in terms of its Hodge  numbers. More precisely the former theorem says that given a compact K\"ahler manifold $(M,h)$ of complex dimension $m=2n$ we have
\begin{equation}
\label{ealora}
\sigma(M)=\sum_{p,q=0}^{m}(-1)^ph^{p,q}_{\overline{\partial}}(M)
\end{equation} 
where $\sigma(M)$ is the signature of $M$, that is the signature of the intersection form of $M$ in middle degree. The latter theorem, that from now on will be called Fr\"olicher index theorem for the sake of simplicity, says that given a compact complex manifold $M$ of complex dimension $m$ its Euler characteristic  satisfies 
\begin{equation}
\label{lamenca}
\chi(M)=\sum_{p,q=0}^{m}(-1)^{p+q}h^{p,q}_{\overline{\partial}}(M).
\end{equation}
The first aim of this note is to reformulate the above two theorems by replacing the terms on the right hand sides of \eqref{ealora} and  \eqref{lamenca} with the corresponding $L^2$-versions. More precisely one of the main result of this paper can be summarized in this way:
\begin{teo}
\label{algarve}
Let $M$ be a compact complex manifold of complex dimension $m$ and let $\pi:\tilde{M}\rightarrow M$ be any Galois $\Gamma$-covering of $M$ where $\Gamma$ stands for the group of deck transformations of $\tilde{M}$. Its Euler characteristic satisfies 
\begin{equation}
\label{adriano}
\chi(M)=\sum_{p,q=0}^{m}(-1)^{p+q}h_{(2),\Gamma,\overline{\partial}}^{p,q}(M).
\end{equation}
Assume now that $m=2n$ and that $M$ admits a K\"ahler metric. Then $\sigma(M)$, the signature of $M$, satisfies 
\begin{equation}
\sigma(M)=\sum_{p,q=0}^{m}(-1)^ph_{(2),\Gamma}^{p,q}(M).
\end{equation}
\end{teo}
We refer to the  first section of this paper for the definition and the main properties of the $L^2$-Hodge numbers $h_{(2),\Gamma,\overline{\partial}}^{p,q}(M)$ and $h_{(2),\Gamma}^{p,q}(M)$. The second main goal of this note is to explore various applications of Th. \ref{algarve}. We are particularly interested in the Euler characteristic of compact, K\"ahler, K\"ahler parabolic manifolds, an interesting class of compact K\"ahler manifolds that includes those with non-positive sectional curvatures, see Def. \ref{garagatto}. Some of our applications can be summarized as follows:
\begin{teo}
\label{appli}
Let $(M,h)$ be a compact, K\"ahler, K\"ahler-parabolic manifold of complex dimension $2m$. Assume that $\sigma(M)\neq 0$. We have the following properties:
\begin{enumerate}
\item  The Euler characteristic of $M$ satisfies $\chi(M)>0$.
\item Let $(\tilde{M},\tilde{h})$ be the universal covering of  $(M,h)$ and let $\tilde{h}:=\pi^*h$. Then there exists at least one pair $(p,q)$ with $p+q=2m$ such that $h^{p,q}_{(2)}(M)\neq 0$. In particular the corresponding space of $L^2$-harmonic forms $\mathcal{H}^{p,q}_{2}(\tilde{M},\tilde{h})$ is infinite dimensional.
\end{enumerate}
\end{teo}
Focusing on compact K\"ahler surfaces we have more applications that can be summarized as follows:
\begin{teo}
\label{rigoberto}
Let $(M,h)$ be a compact K\"ahler surface with infinite fundamental group. Let $\pi:\tilde{M}\rightarrow M$ be the universal covering of $(M,h)$ and let $\tilde{h}:=\pi^*h$. We have the following properties:
\begin{enumerate}
\item Assume that  $\sigma(M)\neq 0$. Then there exists at least one pair $(p,q)$ with $p+q=2$ such that $h^{p,q}_{(2)}(M)\neq 0$. In particular the corresponding space of $L^2$-harmonic forms $\mathcal{H}^{p,q}_{2}(\tilde{M},\tilde{h})$ is infinite dimensional.
\item  Assume  that $\sigma(M)>0$. Then $M$ satisfies $h^{2,0}_{(2)}(M)>0$  and therefore $\mathcal{H}^{2,0}_2(\tilde{M},\tilde{h})$, that is the space of  holomorphic $L^2$-$(2,0)$-forms on $(\tilde{M},\tilde{h})$, is infinite dimensional.
\item  Assume that $h^{1,0}_{(2)}(M)=0$ and that $\sigma(M)\neq 0$. Then $\chi(M)>0$.
\item  Assume that $h^{1,0}_{(2)}(M)=0$ and that $\sigma(M)>0$. Then $\chi(M)>0$ and $\chi(M,\mathcal{O}_M)>0$.
\item Assume that $(M,h)$ has non-positive sectional curvatures and that $\sigma(M)>0$. Then the Kodaira dimension of $M$ is $2$. In particular $M$ is  projective algebraic.
\end{enumerate}
\end{teo}
We continue now  this introduction by describing the structure of this paper. The first section  is divided in two parts. The first subsection contains some generalities about $d$, $\overline{\partial}$ and the corresponding $L^2$-cohomologies. In the second subsection we recall briefly the notion of Von Neumann dimension and then we proceed by defining the $L^2$-Hodge numbers. The  second section is devoted to  Th. \ref{algarve} and its proof. The third  section  is concerned with  the applications of Th. \ref{algarve}, some of them have been previously  summarized in Th. \ref{appli} and Th. \ref{rigoberto}. Finally the fourth and last section of this note collects other properties of the $L^2$-Hodge numbers. We can summarize them as follows:
\begin{teo}
\label{nonsos}
Let $M_1$ and $M_2$ be two compact complex manifolds of complex dimension $m$. Assume that there exists a modification $\phi:M_1\rightarrow M_2$.
Then we have the following equality: $$h^{m,0}_{(2),\overline{\partial}}(M_1)=h^{m,0}_{(2),\overline{\partial}}(M_2).$$
\end{teo}

\begin{cor}
\label{daidaiz}
Let $M_1$ and $M_2$ be two compact complex manifolds of complex dimension $m$. Assume that they are bimeromorphic.
\begin{enumerate}
\item We have the following equality:  $h^{m,0}_{(2),\overline{\partial}}(M_1)=h^{m,0}_{(2),\overline{\partial}}(M_2).$
\item Assume now that $m=2$ and that both $M_1$ and $M_2$ admit a K\"ahler metric. Then $h^{p,0}_{(2)}(M_1)=h^{p,0}_{(2)}(M_2)$ for each $p=0,1,2$ and $h^{0,q}_{(2)}(M_1)=h^{0,q}_{(2)}(M_2)$ for each  $q=0,1,2$. Moreover $h^{1,1}(M_1)=h^{1,1}(M_2)$ if and only if $h^{1,1}_{(2)}(M_1)=h^{1,1}_{(2)}(M_2)$.
\end{enumerate}
\end{cor}
Finally we conclude this introduction with the following remark. It seems reasonable to expect that these results are already known by the experts of this area. On the other hand we were not able to find them in the literature. So we have thought that  writing them down could provide a useful service to the mathematical community.
   
\vspace{1 cm}

\noindent\textbf{Acknowledgments.}  This work was performed within the framework of the LABEX MILYON
(ANR-10-LABX-0070) of Universit\'e de Lyon, within the program "Investissements d'Avenir" (ANR-11-IDEX-0007) operated by the French National Research Agency (ANR).

\section{Background material} 

\subsection{Generalities on $d$, $\overline{\partial}$ and the corresponding $L^2$-cohomologies } 
This section contains a brief summary concerning  some aspects  of the $L^2$-de Rham cohomology and the $L^2$-$\overline{\partial}$-cohomology. We refer to \cite{BruLe} and \cite{Luck} for a thorough discussion of this subject.
Let $(M,h)$ be a complex manifold of complex dimension $m$ endowed with a complete Hermitian metric $h$. In what follows we will label by $\overline{\partial}_{p,q}:\Omega^{p,q}(M)\rightarrow \Omega^{p,q+1}(M)$ the Dolbeault operator acting on $(p,q)$-forms  and with   
 $\overline{\partial}_{p,q}^t:\Omega^{p,q+1}(M)\rightarrow \Omega^{p,q}(M)$  the formal adjoint of $\overline{\partial}_{p,q}$ with respect to the metric $h$.  With $\Delta_{\overline{\partial},p,q}:\Omega^{p,q}(M)\rightarrow \Omega^{p,q}(M)$ we will label the Hodge-Kodaira Laplacian acting on $(p,q)$-forms while with $\overline{\partial}_p+\overline{\partial}_p^t$ we will label the Hodge-Dolbeault operator acting on $\Omega^{p,\bullet}(M)$. $\Delta_{\overline{\partial},p,q}$ is defined as $\Delta_{\overline{\partial},p,q}:= \overline{\partial}_{p,q}^t\circ \overline{\partial}_{p,q}+\overline{\partial}_{p,q-1}\circ \overline{\partial}_{p,q-1}^t$ while $\overline{\partial}_p+\overline{\partial}_p^t:\Omega^{p,\bullet}(M)\rightarrow \Omega^{p,\bullet}(M)$ is the operator given by $\Omega^{p,\bullet}(M):=\bigoplus_{q=0}^m\Omega^{p,q}(M)$ and $\overline{\partial}_p+\overline{\partial}_p^t|_{\Omega^{p,q}(M)}:=\overline{\partial}_{p,q}+\overline{\partial}_{p,q-1}^t$. Consider now the Hilbert space $L^2\Omega^{p,q}(M,h)$. This is the space given by measurable $(p,q)$-forms $\omega$ such that $\langle \omega,\omega\rangle_{L^2\Omega^{p,q}(M,h)}:=\int_Mh(\omega,\omega)\dvol_h<\infty$ where $\dvol_h$ is the volume form manufactured from $h$ and, with a little abuse of notations, we have still labeled by $h$ the Hermitian metric induced by $h$ on $\Lambda^{p,q}(M)$. A well known result says that $\Omega_c^{p,q}(M)$, the space of smooth $(p,q)$-forms with compact support, is dense in $L^2\Omega^{p,q}(M,h)$. Analogously to the case of $\Omega^{p,\bullet}(M)$ we define $L^2\Omega^{p,\bullet}(M,h):=\bigoplus_{q=0}^mL^2\Omega^{p,q}(M,h)$ and $\Omega^{p,\bullet}_c(M):=\bigoplus_{p=0}^m\Omega^{p,q}_c(M)$.  We consider now 
\begin{equation}
\label{primo}
\overline{\partial}_{p,q}:L^2\Omega^{p,q}(M,h)\rightarrow L^2\Omega^{p,q+1}(M,h)
\end{equation}
\begin{equation}
\label{secondo} \overline{\partial}_{p}+\overline{\partial}^t_p:L^2\Omega^{p,\bullet}(M,h)\rightarrow L^2\Omega^{p,\bullet}(M,h)
\end{equation} 
and 
\begin{equation}
\label{terzo}
\Delta_{\overline{\partial},p,q}:L^2\Omega^{p,q}(M,h)\rightarrow L^2\Omega^{p,q}(M,h)
\end{equation}
 as unbounded, closable and densely defined operators defined respectively on $\Omega^{p,q}_c(M)$, $\Omega^{p,\bullet}_c(M)$ and $\Omega^{p,q}_c(M)$. Another well known result tells us  that both \eqref{secondo} and \eqref{terzo} are essentially self-adjoint, see for instance \cite{RoS} or \cite{JW}.  This is actually equivalent to saying that \eqref{secondo} and \eqref{terzo} admit a unique closed extension. Moreover, using the fact that \eqref{secondo} and \eqref{terzo} are essentially self-adjoint, we get  that also \eqref{primo} admits a unique  closed extension, see \cite{BruLe}. Hence from now on with a little abuse of notation we will label with
\begin{equation}
\label{quarto}
\overline{\partial}_{p,q}:L^2\Omega^{p,q}(M,h)\rightarrow L^2\Omega^{p,q+1}(M,h)
\end{equation}
\begin{equation}
\label{quinto} \overline{\partial}_{p}+\overline{\partial}^t_p:L^2\Omega^{p,\bullet}(M,h)\rightarrow L^2\Omega^{p,\bullet}(M,h)
\end{equation} 
 \begin{equation}
\label{sesto}
\Delta_{\overline{\partial},p,q}:L^2\Omega^{p,q}(M,h)\rightarrow L^2\Omega^{p,q}(M,h)
\end{equation}
the unique closed extension of  \eqref{primo}, \eqref{secondo} and \eqref{terzo} respectively. We have an orthogonal decomposition of $L^2\Omega^{p,q}(M,h)$ given by 

\begin{equation}
\label{orto}
L^2\Omega^{p,q}(M,h)=\mathcal{H}^{p,q}_{2,\overline{\partial}}(M,h)\oplus\overline{\im(\overline{\partial}_{p,q})}\oplus\overline{\im({\overline{\partial}^t_{p,q}})}
\end{equation}
where $\mathcal{H}^{p,q}_{2,\overline{\partial}}(M,h)$ is defined as $\mathcal{H}^{p,q}_{2,\overline{\partial}}(M,h):=\ker(\overline{\partial}_{p,q})\cap \ker(\overline{\partial}_{p,q-1}^t)$. It is easy to check that $\mathcal{H}^{p,q}_{2,\overline{\partial}}(M,h)=\ker(\Delta_{\overline{\partial},p,q})$ and  $\overline{\im(\overline{\partial}_{p,q})}\oplus\overline{\im({\overline{\partial}^t_{p,q}})}=\overline{\im(\Delta_{\overline{\partial},p,q})}$.
The reduced $L^2$-Dolbeault cohomology is defined as the quotient
\begin{equation}
\label{redu}
\overline{H}^{p,q}_{2,\overline{\partial}}(M,h):=\frac{\ker(\overline{\partial}_{p,q})}{\overline{\im(\overline{\partial}_{p,q-1})}}.
\end{equation}
The vector space $\overline{H}^{p,q}_{2,\overline{\partial}}(M,h)$ can be infinite dimensional and in general it depends on the  metric $h$. On the other hand it is stable if $h$ is replaced with another metric $k$ which lies in the same quasi-isometry class of $h$, that is $\overline{H}^{p,q}_{2,\overline{\partial}}(M,h)=\overline{H}^{p,q}_{2,\overline{\partial}}(M,k)$ if  $c^{-1}h\leq k\leq ch$ for some constant $c>0$. It is immediate to check from \eqref{orto} that $\overline{H}^{p,q}_{2,\overline{\partial}}(M,h)\cong \mathcal{H}^{p,q}_{2,\overline{\partial}}(M,h)$. Moreover it is another straightforward verification to check that $\ker(\overline{\partial}_p+\overline{\partial}^t_p)=\bigoplus_{q=0}^m\mathcal{H}^{p,q}_{2,\overline{\partial}}(M,h)$. Finally if we replace $\overline{\partial}_{p,q}$ with $\partial_{p,q}$ then we have  analogous definitions and  properties for the operators $\partial_{p,q}:L^2\Omega^{p,q}(M,h)\rightarrow L^2\Omega^{p+1,q}(M,h)$,  $\partial_q+\partial_q^t:L^2\Omega^{\bullet, q}(M,h)\rightarrow L^2\Omega^{\bullet,q}(M,h)$, $\Delta_{\partial,p,q}:L^2\Omega^{p,q}(M,h)\rightarrow L^2\Omega^{p,q}(M,h)$ and the groups $\mathcal{H}^{p,q}_{\partial}(M,h):=\ker(\partial_{p,q})\cap \ker(\partial^t_{p,q-1})$ and $\overline{H}^{p,q}_{2,\partial}(M,h):=\ker(\partial_{p,q})/\overline{\im(\partial_{p-1,q})}$. Now we go on to consider the de Rham differential $d_k:\Omega^k(M)\rightarrow \Omega^{k+1}(M)$. In this case  we will only need to assume that $M$ is a smooth manifold and $h$ is a Riemannian metric on $M$. As in the previous case and with self-explanatory notations we have the formal adjoint of $d_k$ with respect to $h$, $d_k^t:\Omega^{k+1}(M)\rightarrow \Omega^k(M)$, the Hodge-de Rham operator $d+d^t:\Omega^{\bullet}(M)\rightarrow \Omega^{\bullet}(M)$ and the Hodge Laplacian $\Delta_k:\Omega^k(M)\rightarrow \Omega^k(M)$, $\Delta_k=d^t_k\circ d_k+d_{k-1}\circ d^t_{k-1}$. Also in this case we will consider
\begin{equation}
\label{gatto}
d_k:L^2\Omega^k(M,h)\rightarrow L^2\Omega^{k+1}(M,h)
\end{equation}
\begin{equation}
\label{matto} 
d+d^t:L^2\Omega^{\bullet}(M,h)\rightarrow L^2\Omega^{\bullet}(M,h)
\end{equation} 
 \begin{equation}
\label{vero}
\Delta_k:L^2\Omega^{k}(M,h)\rightarrow L^2\Omega^{k}(M,h)
\end{equation}
 as unbounded, closable and densely defined operators defined respectively on $\Omega^{k}_c(M)$, $\Omega^{\bullet}_c(M)$ and $\Omega^{k}_c(M)$ and, analogously to the previous case, we have that \eqref{matto} and \eqref{vero} are essentially self-adjoint and \eqref{gatto} admits a unique closed extension.  Clearly also \eqref{matto} and \eqref{vero} admit a unique closed extension as a consequence of the fact that they are essentially self-adjoint.  Therefore, with a little abuse of notation, from now on with
\begin{equation}
\label{gattox}
d_k:L^2\Omega^k(M,h)\rightarrow L^2\Omega^{k+1}(M,h)
\end{equation}
\begin{equation}
\label{mattox} 
d+d^t:L^2\Omega^{\bullet}(M,h)\rightarrow L^2\Omega^{\bullet}(M,h)
\end{equation} 
 \begin{equation}
\label{verox}
\Delta_k:L^2\Omega^{k}(M,h)\rightarrow L^2\Omega^{k}(M,h)
\end{equation}
we will label the unique closed extensions of \eqref{gatto}, \eqref{matto} and \eqref{vero} respectively. On the space $L^2\Omega^k(M,h)$ we have an orthogonal decomposition given by 
\begin{equation}
\label{ortox}
L^2\Omega^{k}(M,h)=\mathcal{H}^{k}_{2}(M,h)\oplus\overline{\im(d_k)}\oplus\overline{\im(d^t_k)}
\end{equation}
where $\mathcal{H}^{k}_{2}(M,h)$ is defined as $\mathcal{H}^{k}_{2}(M,h):=\ker(d_k)\cap \ker(d^t_{k-1})$. It is easy to check that $\mathcal{H}^{k}_{2}(M,h)=\ker(\Delta_k)$ and  $\overline{\im(d_k)}\oplus\overline{\im(d_k^t)}=\overline{\im(\Delta_k)}$.
Also in this case we define the  reduced $L^2$-de Rham cohomology  as the quotient
\begin{equation}
\label{redux}
\overline{H}^{k}_{2}(M,h):=\frac{\ker(d_k)}{\overline{\im(d_{k-1})}}.
\end{equation}
For the vector space $\overline{H}^{k}_{2}(M,h)$ hold analogous considerations to those made for \eqref{redu}. It can be infinite dimensional and in general it depends on the metric $h$. However if $k$ is another metric such that $c^{-1}h\leq k\leq ch$ for some constant $c>0$ then we have $\overline{H}^{k}_{2}(M,h)=\overline{H}^{k}_{2}(M,k)$.
It is immediate from \eqref{ortox} to check  that $\overline{H}^{k}_{2}(M,h)\cong \mathcal{H}^{k}_{2}(M,h)$ and that $\ker(d+d^t)=\bigoplus_{k=0}^m\mathcal{H}^{k}_{2}(M,h)$.\\
Finally we conclude this section with the following remarks.
In the general case of a complete  Hermitian manifold $(M,h)$ we have a $\mathbb{C}$-anti-linear isomorphism $\mathcal{H}^{p,q}_{2,\overline{\partial}}(M,h)\cong \mathcal{H}^{q,p}_{2,\partial}(M,h)$ induced by the conjugation and a $\mathbb{C}$-linear isomorphism $\mathcal{H}^{p,q}_{2,\overline{\partial}}(M,h)\cong \mathcal{H}^{m-q,m-p}_{2,\partial}(M,h)$ induced by the Hodge star operator. These two isomorphisms together induce the $\mathbb{C}$-anti-linear isomorphism $\mathcal{H}^{p,q}_{2,\overline{\partial}}(M,h)\cong \mathcal{H}^{m-p,m-q}_{2,\overline{\partial}}(M,h)$. Moreover the Hodge star operator induces also a $\mathbb{C}$-linear isomorphism between $\mathcal{H}^k_2(M,h)$ and $\mathcal{H}^{2m-k}_2(M,h)$. Furthermore if we assume that $h$ is K\"ahler then we have a decomposition $\mathcal{H}^{k}_{2}(M,h)= \bigoplus_{p+q=k}\mathcal{H}^{p,q}_{2,\overline{\partial}}(M,h)$ and  a $\mathbb{C}$-anti-linear isomorphism  $\mathcal{H}^{p,q}_{2,\overline{\partial}}(M,h)\cong \mathcal{H}^{q,p}_{2,\overline{\partial}}(M,h)$,  the latter as a consequence of the equality $\mathcal{H}^{p,q}_{2,\overline{\partial}}(M,h)=\mathcal{H}^{p,q}_{2,\partial}(M,h)$ and the $\mathbb{C}$-anti-linear isomorphism $\mathcal{H}^{p,q}_{2,\overline{\partial}}(M,h)\cong \mathcal{H}^{q,p}_{2,\partial}(M,h)$ induced by the conjugation. These properties  follow by the fact that in the K\"ahler case we have  $\Delta_k=2\Delta_{\overline{\partial},p,q}$  and $\Delta_{\overline{\partial},p,q}=\Delta_{\partial,p,q}$ on $\Omega_c^{p,q}(M)$. Therefore, as unbounded and  closed operators acting on $L^2\Omega^k(M,h)$ and on $L^2\Omega^{p,q}(M,h)$, we have $\Delta_k=2\bigoplus_{p+q=k}\Delta_{\overline{\partial},p,q}$ and  $\Delta_{\partial,p,q}=\Delta_{\overline{\partial},p,q}$ respectively. Hence, when $(M,h)$ is K\"ahler, we can simply write $\mathcal{H}^{p,q}_{2}(M,h)$ and $\overline{H}^{p,q}_2(M,h)$ instead of $\mathcal{H}^{p,q}_{2, \overline{\partial}}(M,h)$ and $\overline{H}^{p,q}_{2,\overline{\partial}}(M,h)$.

\subsection{Coverings and Von-Neumann dimension}
In this section we recall the definition of the $L^2$-Betti numbers and $L^2$-Hodge numbers. We introduce only what is strictly necessary for our purposes without any goal of completeness. For an in-depth treatment of this topic we refer to \cite{Luck}. We also invite the reader to consult the seminal paper of Atiyah \cite{MA}. Finally for a quick introduction we refer to \cite{MaMa} and to \cite{Roe}.\\
Let $(M,g)$ be a compact Riemannian  manifold of dimension $m$. Let $\pi:\tilde{M}\rightarrow M$ be a Galois covering of $M$ and let $\tilde{g}:=\pi^*g$ be the pull-back metric on $\tilde{M}$. Let $\Gamma$ be the group of deck transformations acting on $\tilde{M}$, $\Gamma\times \tilde{M}\rightarrow \tilde{M}$. We recall that  $\Gamma$ is a discrete group whose action on $\tilde{M}$ preserves the fibers of $\tilde{M}$, that is $\pi(\gamma(\tilde{p}))=\pi(\tilde{p})$ for each $\tilde{p}\in \pi^{-1}(p)$ and for each $\gamma\in \Gamma$. Moreover the action is transitive on all fibers and properly discontinuous and we have $\tilde{M}/\Gamma=M$. In what follows we will simply summarize this data by saying that $\pi:\tilde{M}\rightarrow M$ is a Galois $\Gamma$-covering. An open subset $A\subset \tilde{M}$ is a  fundamental domain of the action of $\Gamma$ on $\tilde{M}$ if
\begin{itemize}
\item $\tilde{M}=\bigcup_{\gamma\in \Gamma}\gamma(\overline{A})$, 
\item $\gamma_1(A)\cap\gamma_2(A)=\emptyset$ for every $\gamma_1, \gamma_2\in \Gamma$ with $\gamma_1\neq \gamma_2$,
\item $\overline{A}\setminus A$ has zero measure.
\end{itemize}   
It is not difficult to see that $L^2\Omega^{k}(\tilde{M},\tilde{g})\cong L^2\Gamma \otimes L^2\Omega^{k}(A,\tilde{g}|_A)\cong L^2\Gamma\otimes L^2\Omega^{k}(M,g)$ where a basis for $L^2\Gamma$ is given by the functions $\delta_{\gamma}$ with $\gamma\in \Gamma$ defined as $\delta_{\gamma}(\gamma')=1$ if $\gamma=\gamma'$ and $\delta_{\gamma}(\gamma')=0$ if $\gamma\neq \gamma'$. Moreover it is clear that $\Gamma$ acts on $L^2\Omega^k(\tilde{M},\tilde{g})$ by isometries.  Consider now a $\Gamma$-module $V\subset L^2\Omega^{k}(\tilde{M},\tilde{g})$, that is a closed subspace of $L^2\Omega^{k}(\tilde{M},\tilde{g})$ which is invariant under the action of $\Gamma$. If $\{\eta_j\}_{j\in \mathbb{N}}$ is an orthonormal basis for $V$ then the following quantity is finite $$\sum_{j\in \mathbb{N}}\int_A\tilde{g}(\eta_j,\eta_j)|_A\dvol_{\tilde{g}|_A}$$ and does not depend either on the choice of the orthonormal basis of $V$ or on the choice of the fundamental domain of the action of $\Gamma$ on $\tilde{M}$. Therefore if $B$ is another fundamental domain and $\{\alpha_j\}_{j\in \mathbb{N}}$ is another orthonormal basis for $V$ we have $\sum_{j\in \mathbb{N}}\int_A\tilde{g}(\eta_j,\eta_j)|_A\dvol_{\tilde{g}|_A}=\sum_{j\in \mathbb{N}}\int_B\tilde{g}(\alpha_j,\alpha_j)|_B\dvol_{\tilde{g}|_B}$. The Von-Neumann dimension of a $\Gamma$-module $V$ is therefore defined as $$\dim_{\Gamma}(V):=\sum_{j\in \mathbb{N}}\int_A\tilde{g}(\eta_j,\eta_j)|_A\dvol_{\tilde{g}|_A}$$ where $\{\eta_j\}_{j\in \mathbb{N}}$ is any orthonormal basis for $V$ and $A$ is any fundamental domain of the action of $\Gamma$ on $\tilde{M}$. The previous considerations show that the above definition is well posed. Since the Hodge Laplacian $\Delta_{k}:L^2\Omega^{k}(\tilde{M},\tilde{g})\rightarrow L^2\Omega^{k}(\tilde{M},\tilde{g})$ commutes with the action of $\Gamma$, a natural and important example of $\Gamma$-module is provided by the space of $L^2$ harmonic forms of degree $k$, $\mathcal{H}^{k}_{2}(\tilde{M},\tilde{g})$, for each $k=0,...,m$. Analogously if $(M,h)$ is a compact complex Hermitian manifold of complex dimension $m$, $\pi:\tilde{M}\rightarrow M$ is a Galois $\Gamma$-covering  and $\tilde{h}=\pi^*h$ is the pull-back metric on $\tilde{M}$ we have $L^2\Omega^{p,q}(\tilde{M},\tilde{h})\cong L^2\Gamma \otimes L^2\Omega^{p,q}(A,\tilde{h}|_A)\cong L^2\Gamma\otimes L^2\Omega^{p,q}(M,h)$ and besides the $\Gamma$-modules $\mathcal{H}^{k}_2(\tilde{M},\tilde{h})$ we have also the $\Gamma$-modules $\mathcal{H}^{p,q}_{2,\overline{\partial}}(\tilde{M},\tilde{h})$.\\ We recall now the definition of  the $L^2$-Betti numbers and, in the complex Hermitian setting, that  of $L^2$-Hodge numbers. The $L^2$-Betti numbers of a compact Riemannian manifold $(M,g)$ with respect to a Galois $\Gamma$-covering $\pi:\tilde{M}\rightarrow M$ endowed with the pull-back metric $\tilde{g}:=\pi^*g$ are defined as $$b^{(2)}_{k,\Gamma}(M):=\dim_{\Gamma}(\mathcal{H}_2^k(\tilde{M},\tilde{g}))$$ for $k=0,...,m$. The $L^2$-Betti numbers $b^{(2)}_{k,\Gamma}(M)$ are nonnegative real numbers and their value does not depend on the choice of the Riemannian metric on $M$. Likewise if $(M,h)$ is a compact complex Hermitian manifold of complex dimension $m$, $\pi:\tilde{M}\rightarrow M$ is a Galois $\Gamma$-covering and $\tilde{h}:=\pi^*h$ we have the $L^2$-Hodge numbers defined as $$h^{p,q}_{(2),\Gamma,\overline{\partial}}(M):=\dim_{\Gamma}(\mathcal{H}^{p,q}_{2,\overline{\partial}}(\tilde{M},\tilde{h})).$$ Also in this case $h^{p,q}_{(2),\Gamma,\overline{\partial}}(M)$ are nonnegative real numbers independent on the choice of the Hermitian metric on $M$. We recall now two deep results which are both a particular case of the celebrated Atiyah's $L^2$-index theorem, see \cite{MA}. For the first theorem below  see \cite{MA} pag. 71. The second one, although not explicitly stated in \cite{MA}, follows by using, in setting of Hermitian manifolds, the same strategy used to derive Th. \ref{gamma1} in that of Riemannian manifolds. 

\begin{teo}
\label{gamma1}
Let $(M,g)$ be a compact Riemannian manifold of  dimension $m$. Let $\pi:\tilde{M}\rightarrow M$ be a Galois $\Gamma$-covering of $M$ and let $\tilde{g}:=\pi^{*}g$ be the pull-back metric on $\tilde{M}$. Then we have $$\sum_{q=0}^m(-1)^qb_k(M)=\sum_{q=0}^m(-1)^qb_{k,\Gamma}^{(2)}(M,g).$$
\end{teo}

\begin{teo}
\label{gamma2}
Let $(M,h)$ be a compact complex Hermitian manifold of complex dimension $m$. Let $\pi:\tilde{M}\rightarrow M$ be a Galois $\Gamma$-covering of $M$ and let $\tilde{h}:=\pi^{*}h$ be the pull-back metric on $\tilde{M}$. Then for any fixed $p\in \{0,...,m\}$ we have $$\sum_{q=0}^m(-1)^qh_{\overline{\partial}}^{p,q}(M)=\sum_{q=0}^m(-1)^qh_{(2),\Gamma,\overline{\partial}}^{p,q}(M,h).$$
\end{teo}

\noindent We add now some observations. Let $(M,g)$ be a compact Riemannian manifold with infinite fundamental group. Let $\pi:\tilde{M}\rightarrow M$ be a Galois $\Gamma$-covering with $\Gamma$ infinite too. Then we have $b^{(2)}_{0,\Gamma}(M)=0$. Indeed let $p$ be any point in $M$. Let $U$ be a sufficiently small open neighborhood of $p$ such that $\pi^{-1}(U)=\cup_{n\in\mathbb{N}}U_n$, $U_i\cap U_j=\emptyset$ for $i\neq j$ and $\pi|_{U_n}:U_n\rightarrow U$ is a diffeomorphism. This in turn implies that  $\pi|_{U_n}:(U_n,\tilde{g}|_{U_n})\rightarrow (U,g|_U)$ is an isometry and therefore we have $\vol_{\tilde{g}}(\tilde{M})\geq \sum_{n\in \mathbb{N}}\vol_{\tilde{g}}(U_n)=\sum_{n\in \mathbb{N}}\vol_{g}(U)=\infty$. Hence $\vol_{\tilde{g}}(\tilde{M})=\infty$. It is now a classical result of global analysis that the only $L^2$-harmonic function on a complete Riemannian manifold of infinite volume is $0$, see for instance \cite{Dod} Th. 1 or \cite{RoS} pag. 59. In conclusion $\mathcal{H}^0_2(\tilde{M},\tilde{g})=0$ and so $b^{(2)}_{0,\Gamma}(M)=0$. Consider now the case of  a compact complex Hermitian manifold $(M,h)$ of complex dimension $m$. Let $\pi:\tilde{M}\rightarrow M$ be a Galois $\Gamma$-covering of $M$ and let $\tilde{h}:=\pi^*h$. As recalled in the previous section we have $\mathbb{C}$-linear isomorphism $\mathcal{H}^k_2(\tilde{M},\tilde{h})\cong \mathcal{H}^{2m-k}_2(\tilde{M},\tilde{h})$ induced by the Hodge star operator and a $\mathbb{C}$-anti-linear isomorphism $\mathcal{H}_{2,\overline{\partial}}^{p,q}(\tilde{M},\tilde{h})\cong \mathcal{H}_{2,\overline{\partial}}^{m-p,m-q}(\tilde{M},\tilde{h})$ obtained by taking the composition of the Hodge star operator and the conjugation. Hence at the level of $L^2$-Betti numbers and  $L^2$-Hodge numbers we get the following equalities:
\begin{equation}
\label{hhh}
b^{(2)}_{k,\Gamma}(M)=b_{2m-k,\Gamma}^{(2)}(M),\ h^{p,q}_{(2),\overline{\partial},\Gamma}(M)=h^{m-p,m-q}_{(2),\overline{\partial},\Gamma}(M).
\end{equation}
When $(M,h)$ is K\"ahler, we have $\mathcal{H}^{p,q}_{2,\overline{\partial}}(\tilde{M},\tilde{h})=\mathcal{H}^{p,q}_{2,\partial}(\tilde{M},\tilde{h})$, $\mathcal{H}^{k}_{2}(\tilde{M},\tilde{h})= \bigoplus_{p+q=k}\mathcal{H}^{p,q}_{2,\overline{\partial}}(\tilde{M},\tilde{h})$ and a $\mathbb{C}$-anti-linear isomorphism $\mathcal{H}^{p,q}_{2,\overline{\partial}}(\tilde{M},\tilde{h})\cong \mathcal{H}^{q,p}_{2,\overline{\partial}}(\tilde{M},\tilde{h})$ that follows by the equality $\mathcal{H}^{p,q}_{2,\overline{\partial}}(\tilde{M},\tilde{h})=\mathcal{H}^{p,q}_{2,\partial}(\tilde{M},\tilde{h})$ and the $\mathbb{C}$-anti-linear isomorphism $\mathcal{H}^{p,q}_{2,\overline{\partial}}(\tilde{M},\tilde{h})\cong \mathcal{H}^{q,p}_{2,\partial}(\tilde{M},\tilde{h})$  induced by the conjugation. Therefore in the K\"ahler setting we can simply write $h^{p,q}_{(2),\Gamma}(M)$ for the $L^2$-Hodge numbers of $(M,h)$ and, in addition to \eqref{hhh}, we have the following properties:
\begin{equation}
\label{please}
b_{k,\Gamma}^{(2)}(M)=\bigoplus_{p+q=k}h^{p,q}_{(2),\Gamma}(M),\ h^{p,q}_{(2),\Gamma}(M)=h^{q,p}_{(2),\Gamma}(M).
\end{equation}
We conclude  this section with the following remark about the notation: if $\pi:\tilde{M}\rightarrow M$ is the universal covering of $M$ then we will simply write $b^{(2)}_k(M)$ instead of $b^{(2)}_{k,\Gamma}(M)$ for the $L^2$-Betti numbers of $M$ with respect to $\pi:\tilde{M}\rightarrow M$. Analogously if $\pi:\tilde{M}\rightarrow M$ is the universal covering of a compact complex manifold $M$ then we will write $h^{p,q}_{(2),\overline{\partial}}(M)$ instead of $h^{p,q}_{(2),\Gamma, \overline{\partial}}(M)$ for the $L^2$-Hodge numbers of $M$ with respect to $\pi:\tilde{M}\rightarrow M$. Finally if $M$ admits a K\"ahler metric then we will simply write $h^{p,q}_{(2)}(M)$ for the $L^2$-Hodge numbers of $M$ with respect to its universal covering  $\pi:\tilde{M}\rightarrow M$.

\section{Hodge and Fr\"olicher index theorems via Von Neumann dimension}
This section contains two of the main results of this paper. In the first theorem we provide  a reformulation of the classical Hodge index theorem  by replacing the Hodge numbers with  the $L^2$-Hodge numbers. Similarly the second result  is a reformulation of the classical Fr\"olicher index theorem made again by using the $L^2$-Hodge numbers.

\begin{teo}
\label{gammahodge}
Let $(M,h)$ be a compact K\"ahler manifold of complex dimension $m=2n$. Let $\pi:\tilde{M}\rightarrow M$ be a Galois $\Gamma$-covering of $M$ and let $\tilde{h}:=\pi^*h$ be the pull-back metric on $\tilde{M}$. Then $$\sigma(M)=\sum_{p,q=0}^{m}(-1)^ph_{(2),\Gamma}^{p,q}(M).$$
\end{teo}

\begin{proof}
Consider the well known Hodge index theorem:$$\sigma(M)=\sum_{p,q=0}^{m}(-1)^ph^{p,q}(M).$$ We can rewrite the above formula as 
\begin{align}
 \sigma(M)= & h^{0,0}(M)-h^{1,0}(M)+h^{2,0}(M)-....+h^{m,0}(M)+\\
\nonumber & h^{0,1}(M)-h^{1,1}(M)+h^{2,1}(M)-....+h^{m,1}(M)+...\\
\nonumber &...+h^{0,m}(M)-h^{1,m}(M)+h^{2,m}(M)-....+h^{m,m}(M)=
\end{align}
using the fact that $h^{p,q}(M)=h^{q,p}(M)$ since  $(M,h)$ is K\"ahler,
\begin{align}
\nonumber & h^{0,0}(M)-h^{0,1}(M)+h^{0,2}(M)-....+h^{0,m}(M)+\\
\nonumber & h^{1,0}(M)-h^{1,1}(M)+h^{1,2}(M)-....+h^{1,m}(M)+...\\
\nonumber &...+h^{m,0}(M)-h^{m,1}(M)+h^{m,2}(M)-....+h^{m,m}(M)
\end{align}
that is $$\sigma(M)=\sum_{q=0}^{m}(-1)^qh^{0,q}(M)+\sum_{q=0}^{m}(-1)^qh^{1,q}(M)+...+\sum_{q=0}^{m}(-1)^qh^{m,q}(M).$$ Using Th. \ref{gamma2} we have $$\sigma(M)=\sum_{q=0}^{m}(-1)^qh_{(2),\Gamma}^{0,q}(M)+\sum_{q=0}^{m}(-1)^qh_{(2),\Gamma}^{1,q}(M)+...+\sum_{q=0}^{m}(-1)^qh^{m,q}_{(2),\Gamma}(M)$$
and by the fact that $h^{p,q}_{(2),\Gamma}(M)=h^{q,p}_{(2),\Gamma}(M)$ we get that  $$\sigma(M)=\sum_{p=0}^{m}(-1)^ph_{(2),\Gamma}^{p,0}(M)+\sum_{p=0}^{m}(-1)^ph^{p,1}_{(2),\Gamma}(M)+...+\sum_{p=0}^{m}(-1)^ph^{p,m}_{(2),\Gamma}(M)$$ that we can rewrite as   $$\sigma(M)=\sum_{p,q=0}^{m}(-1)^ph^{p,q}_{(2),\Gamma}(M).$$ The proof is thus complete.
\end{proof}

We come now to the other result of this section.

\begin{teo}
\label{gammafrol}
Let $(M,h)$ be a compact complex Hermitian manifold of complex dimension $m$. Let $\pi:\tilde{M}\rightarrow M$ be a Galois $\Gamma$-covering of $M$ and let $\tilde{h}:=\pi^*h$ be the pull-back metric on $\tilde{M}$. Then $$\chi(M)=\sum_{p,q=0}^{m}(-1)^{p+q}h_{(2),\Gamma,\overline{\partial}}^{p,q}(M).$$
\end{teo}

\begin{proof}
Consider the well known Fr\"olicher theorem:$$\chi(M)=\sum_{p,q=0}^{m}(-1)^{p+q}h_{\overline{\partial}}^{p,q}(M).$$ We can rewrite the above formula as 
\begin{align}
 \chi(M)= & h^{0,0}_{\overline{\partial}}(M)-h^{0,1}_{\overline{\partial}}(M)+h^{0,2}_{\overline{\partial}}(M)-....+(-1)^mh^{0,m}_{\overline{\partial}}(M)+\\
\nonumber & -h^{1,0}_{\overline{\partial}}(M)+h^{1,1}_{\overline{\partial}}(M)-h^{1,2}_{\overline{\partial}}(M)+....+(-1)^{m+1}h^{1,m}_{\overline{\partial}}(M)+...\\
\nonumber &...+(-1)^mh^{m,0}_{\overline{\partial}}(M)+(-1)^{m+1}h^{m,1}_{\overline{\partial}}(M)+(-1)^{m+2}h^{m,2}_{\overline{\partial}}(M)+....+h^{m,m}_{\overline{\partial}}(M)
\end{align}
that is $$\chi(M)=\sum_{q=0}^{m}(-1)^qh^{0,q}_{\overline{\partial}}(M)-\sum_{q=0}^{m}(-1)^qh^{1,q}_{\overline{\partial}}(M)+\sum_{q=0}^{m}(-1)^qh^{2,q}_{\overline{\partial}}(M)+...+(-1)^m\sum_{q=0}^{m}(-1)^qh^{m,q}_{\overline{\partial}}(M).$$ Using Th. \ref{gamma2} we have $$\chi(M)=\sum_{q=0}^{m}(-1)^qh^{0,q}_{(2),\Gamma,\overline{\partial}}(M)-\sum_{q=0}^{m}(-1)^qh^{1,q}_{(2),\Gamma,\overline{\partial}}(M)+\sum_{q=0}^{m}(-1)^qh^{2,q}_{(2),\Gamma,\overline{\partial}}(M)+...+(-1)^m\sum_{q=0}^{m}(-1)^qh^{m,q}_{(2),\Gamma,\overline{\partial}}(M)$$
that we can rewrite as   $$\chi(M)=\sum_{p,q=0}^{m}(-1)^{p+q}h^{p,q}_{(2),\Gamma,\overline{\partial}}(M).$$ The proof is thus complete.
\end{proof}

\section{Applications}
This section is concerned with some applications of Th. \ref{gammahodge} and Th. \ref{gammafrol}. We are mainly interested in the sign of the Euler characteristic of compact non-positively curved K\"ahler manifolds. This is a particular case of an old problem that goes back to Hopf, see \cite{Luck} and the references cited there. Indeed   Hopf conjectured that for a compact, even dimensional Riemannian manifold $(M,g)$ it holds $(-1)^m\chi(M)\geq 0$ if $\sec_g(M)\leq 0$ and $(-1)^m\chi(M)>0$ if $\sec_g(M)<0$ where $2m=\dim(M)$ and $\sec_g(M)$ are the sectional curvatures of $(M,g)$. This conjecture, which is still open, has been completely settled in the setting of K\"ahler manifolds. More precisely Gromov proved in \cite{Gromov} that $\sec_g(M)<0$ implies $(-1)^m\chi(M)>0$ and Cao-Xavier in \cite{CaoXa} and Jost-Zuo in \cite{Jost} showed that $(-1)^m\chi(M)\geq 0$ when  $\sec_g(M)\leq 0$.  We recall now some of the main definitions and results from \cite{CaoXa}, \cite{Gromov} and \cite{Jost}.

\begin{defi}
Let $(M,h)$ be a compact K\"ahler manifold and let $\omega$ be the corresponding $(1,1)$-form. Let $\pi:\tilde{M}\rightarrow M$ be the universal covering of $M$. Let $\tilde{h}:=\pi^*h$ and $\tilde{\omega}:=\pi^*\omega$. Then $(M,h)$ is said K\"ahler hyperbolic if $\tilde{\omega}$ is $d$-bounded, that is if there exists a $1$-form $\eta\in \Omega^1(\tilde{M})\cap L^{\infty}\Omega^1(\tilde{M},\tilde{h})$ such that $d\eta=\tilde{\omega}$.
\end{defi}

In \cite{Gromov} Gromov proved the following:

\begin{teo}
\label{hyper}
Let $(M,h)$ be a compact K\"ahler manifold of complex dimension $m$:
\begin{itemize}
\item assume that   $(M,h)$ is K\"ahler hyperbolic.  Let $\pi:\tilde{M}\rightarrow M$ be the universal covering of $M$ and let $\tilde{h}:=\pi^*h$. Then $\mathcal{H}_2^{p,q}(\tilde{M},\tilde{h})=0$ if $p+q\neq m$ and  $\mathcal{H}_2^{p,q}(\tilde{M},\tilde{h})\neq 0$ if $p+q=m$. As a consequence we have $(-1)^m\chi(M)>0$,
\item  if $\sec_h(M)<0$ then  $(M,h)$ is K\"ahler hyperbolic.
\end{itemize}
\end{teo}

In \cite{CaoXa} and \cite{Jost} the authors introduced the next definition, which includes as a particular case, the above definition of Gromov:
\begin{defi}
\label{garagatto}
Let $(M,h)$ be a compact K\"ahler manifold and let $\omega$ be the corresponding $(1,1)$-form. Let $\pi:\tilde{M}\rightarrow M$ be the universal covering of $M$. Let $\tilde{h}:=\pi^*h$ and $\tilde{\omega}:=\pi^*\omega$. Then $(M,h)$ is said K\"ahler parabolic \footnote{K\"ahler non elliptic in \cite{Jost}.} if $\tilde{\omega}$ is $d$-sublinear, that is if there exists a $1$-form $\eta\in \Omega^1(\tilde{M})$, a point $q\in \tilde{M}$ and  constants $b>0$, $c>0$ such that $d\eta=\tilde{\omega}$ and $|\eta|_{\tilde{h}}(x)\leq bd_{\tilde{h}}(q,x)+c$ where $d_{\tilde{h}}$ is the distance function associated to $\tilde{h}$ and $|\eta|_{\tilde{h}}$ is the pointwise norm of $\eta$.
\end{defi}

In this setting the authors proved in \cite{CaoXa} and \cite{Jost} the following:
\begin{teo}
\label{para}
Let $(M,h)$ be a compact K\"ahler manifold of complex dimension $m$:
\begin{itemize}
\item assume that   $(M,h)$ is K\"ahler parabolic.  Let $\pi:\tilde{M}\rightarrow M$ be the universal covering of $M$ and let $\tilde{h}:=\pi^*h$. Then $\mathcal{H}_2^{p,q}(\tilde{M},\tilde{h})=0$ if $p+q\neq m$. As a consequence we have $(-1)^m\chi(M)\geq 0$,
\item  if $\sec_h(M)\leq 0$ then  $(M,h)$ is K\"ahler parabolic.
\end{itemize}
\end{teo}

\noindent As we can see Th. \ref{para} extends to the K\"ahler parabolic case the vanishing result contained in Th. \ref{hyper} concerning the K\"ahler hyperbolic case. On the other hand Th. \ref{para} does not provide any information about the non-vanishing of $\mathcal{H}^{p,q}_2(\tilde{M},\tilde{h})$ when $p+q=m$ or equivalently about the non-vanishing of $\chi(M)$. Therefore some natural questions that arise comparing Th. \ref{hyper} and Th. \ref{para}  are the following:
\begin{itemize}
\item let $(M,h)$ be a compact K\"ahler, K\"ahler parabolic manifold. Under what circumstances does  the Euler characteristic of $M$ vanish? 
\item  let $(M,h)$ be a compact K\"ahler, K\"ahler parabolic manifold, let $\pi:\tilde{M}\rightarrow M$ be its universal covering and let $\tilde{h}:=\pi^*h$. Under what circumstances is $\mathcal{H}^{p,m-p}_2(\tilde{M},\tilde{h})\neq 0?$ In particular when do we have $\mathcal{H}^{m,0}_2(\tilde{M},\tilde{h})\neq 0?$
\end{itemize}
In what follows we will show that these questions are deeply connected with the non vanishing of the signature of $M$.

\begin{prop}
\label{salomone}
Let $(M,h)$ be a compact, K\"ahler, K\"ahler-parabolic manifold of complex dimension $2m$. Assume that $\sigma(M)\neq 0$. Then $\chi(M)>0$.
\end{prop}

\begin{proof}
Let $\pi:\tilde{M}\rightarrow M$ be the universal covering of $M$ and let $\tilde{h}:=\pi^*h$. According to Th.\ref{gamma1} we have $\chi(X)=\sum_{k=0}^{4m}(-1)^kb^{(2)}_k(M)$. As shown in \cite{CaoXa} and \cite{Jost} we have $b^{(2)}_k(M)=0$ whenever $k\neq 2m$. Therefore, by \eqref{please}, we get $h^{p,q}_{(2)}(M)=0$ whenever $p+q\neq 2m$. Hence using Th. \ref{gammahodge} we have 
\begin{equation}
\label{riccione}
0\neq \sigma(M)=\sum_{p,q=0}^{2m}(-1)^{p}h^{p,q}_{(2)}(M)=\sum_{p=0}^{2m}(-1)^{p}h^{p,2m-p}_{(2)}(M).
\end{equation}
Clearly \eqref{riccione} implies the existence of  at least one pair $(p,2m-p)$ such that $h^{p,2m-p}_{(2)}(M)\neq 0$. As a consequence we deduce that $b^{(2)}_{2m}(M)\neq 0$ because, as recalled in \eqref{please}, we have $b^{(2)}_k(M)=\bigoplus_{p+q=k}h^{p,q}_{(2)}(M)$. Finally, as we have $\chi(X)=\sum_{k=0}^{4m}(-1)^kb^{(2)}_k(M)=b^{(2)}_{2m}(M)$, we can conclude that $\chi(X)>0$ as desired.
\end{proof}

\begin{cor}
\label{casetta}
Let $(M,h)$ be a compact, K\"ahler, K\"ahler-parabolic manifold of complex dimension $2m$. Assume that $\sigma(M)\neq 0$. Let $(\tilde{M},\tilde{h})$ be the universal covering of  $(M,h)$ and let $\tilde{h}:=\pi^*h$. Then there exists at least one pair $(p,q)$ with $p+q=2m$ such that $h^{p,q}_{(2)}(M)\neq 0$. Hence the corresponding space of $L^2$-harmonic forms $\mathcal{H}^{p,q}_{2}(\tilde{M},\tilde{h})$ is infinite dimensional.
\end{cor}
\begin{proof}
The first part of this corollary is contained in the proof of Prop. \ref{salomone}. The second one follows by \cite{Roe} Lemma 15.10.
\end{proof}

\begin{cor}
\label{cinzia}
Let $(M,h)$ be a compact K\"ahler  manifold of complex dimension $2m$ with non positive sectional curvatures. Assume that $\sigma(M)\neq 0$. Then Prop. \ref{salomone} and Cor. \ref{casetta} hold true for $(M,h)$.  
\end{cor}

\begin{proof}
This follows by the fact that if a compact  K\"ahler manifold $(M,h)$ has non positive sectional curvatures then it is K\"ahler parabolic. See for instance \cite{CaoXa} or \cite{Jost}.
\end{proof}

From now on we specialize to the setting of compact complex surfaces.

\begin{prop}
\label{tianjin}
Let $(M,h)$ be a compact K\"ahler surface with infinite fundamental group and $\sigma(M)\neq 0$. Let $\pi:\tilde{M}\rightarrow M$ be the universal covering of $(M,h)$ and let $\tilde{h}:=\pi^*h$. Then:
\begin{enumerate}
\item There exists at least one pair $(p,q)$ with $p+q=2$ such that $h^{p,q}_{(2)}(M)\neq 0$. In particular the corresponding space of $L^2$-harmonic forms $\mathcal{H}^{p,q}_{2}(\tilde{M},\tilde{h})$ is infinite dimensional.
\item  Assume now that $\sigma(M)>0$. Then $M$ satisfies $2h^{2,0}_{(2)}(M)>h^{1,1}_{(2)}(M)$. In particular $h^{2,0}_{(2)}(M)> 0$ and therefore $\mathcal{H}^{2,0}_2(\tilde{M},\tilde{h})$, that is the space of  holomorphic $L^2$-$(2,0)$-forms on $(\tilde{M},\tilde{h})$, is infinite dimensional. 
\end{enumerate}
\end{prop}

\begin{proof}
Let $\pi:\tilde{M}\rightarrow M$ be the universal covering of $M$ and let $\tilde{h}:=\pi^*h$. According to the first remark after Th. \ref{gamma2} we know that $(\tilde{M},\tilde{h})$ has infinite volume and this in turn implies that $0=\mathcal{H}^{0,0}_{2}(\tilde{M},\tilde{h})$. Since $\mathcal{H}^{0,0}_{2}(\tilde{M},\tilde{h})\cong \mathcal{H}^{2,2}_{2}(\tilde{M},\tilde{h})$ we can also conclude that $\mathcal{H}^{2,2}_{2}(\tilde{M},\tilde{h})=0$. Moreover, according to \eqref{hhh} and \eqref{please}, we have  $h^{1,0}_{(2)}(M)=h^{0,1}_{(2)}(M)=h^{2,1}_{(2)}(M)=h^{1,2}_{(2)}(M)$   and $h^{2,0}_{(2)}(M)=h^{0,2}_{(2)}(M)$. This yields the following simplification within the formula proved in Th. \ref{gammahodge}:
\begin{equation}
\label{sobsob}
\sigma(M)=\sum_{p,q=0}^{2}(-1)^ph_{(2)}^{p,q}(M)=2h^{2,0}_{(2)}(M)-h^{1,1}_{(2)}(M)
\end{equation}
 and so we can deduce that $2h^{2,0}_{(2)}(M)-h^{1,1}_{(2)}(M)\neq 0$. Therefore there exists at least one pair $(p,q)$ with $p+q=2$ such that $h^{p,q}_{(2)}(M)\neq 0$. Applying  Lemma 15.10 in \cite{Roe} we can conclude that $\mathcal{H}^{p,q}_{2}(\tilde{M},\tilde{h})$ is infinite dimensional. If $\sigma(M)>0$ then $2h^{2,0}_{(2)}(M)-h^{1,1}_{(2)}(M)>0$ and therefore $h^{2,0}_{(2)}(M)>0$. Applying again Lemma 15.10 in \cite{Roe} we can conclude that $\mathcal{H}^{2,0}_{2}(\tilde{M},\tilde{h})$ is infinite dimensional. 
\end{proof}

Assuming moreover that $h^{1,0}_{(2)}(M)=0$ we have the following version of Prop. \ref{salomone}. Clearly the assumption $h^{1,0}_{(2)}(M)=0$ is satisfied when $(M,h)$ is a  K\"ahler parabolic surface.

\begin{prop}
\label{luccio}
Let $(M,h)$ be a compact K\"ahler surface with infinite fundamental group and  $h^{1,0}_{(2)}(M)=0$. We have the following properties:
\begin{enumerate}
 \item if  $\sigma(M)\neq 0$ then $\chi(M)>0$,
\item if $\sigma(M)>0$ then $\chi(M)>0$ and $\chi(M,\mathcal{O}_M)>0$.
\end{enumerate}
\end{prop}

\begin{proof}
Let $\pi:\tilde{M}\rightarrow M$ be the universal covering of $M$ and let $\tilde{h}:=\pi^*h$. As in Prop. \ref{tianjin}  we have $$\sigma(M)=\sum_{p,q=0}^{2}(-1)^ph_{(2)}^{p,q}(M)=2h^{2,0}_{(2)}(M)-h^{1,1}_{(2)}(M)$$ and so we can deduce that $2h^{2,0}_{(2)}(M)-h^{1,1}_{(2)}(M)\neq 0$. Using  the fact that $(\tilde{M},\tilde{h})$ has infinite volume we have $0=\mathcal{H}^0_2(\tilde{M},\tilde{h})\cong \mathcal{H}^4_2(\tilde{M},\tilde{h})$, which implies $0=b^{(2)}_{0}(M)=b^{(2)}_4(M)$. Moreover, by the  the fact that $h^{1,0}_{(2)}(M)=h^{0,1}_{(2)}(M)$, $b^{(2)}_1(M)=b^{(2)}_3(M)$ and $b_1^{(2)}(M)=h^{1,0}_{(2)}(M)+h^{0,1}_{(2)}(M)$, we obtain the vanishing of both $b^{(2)}_1(M)$ and $b^{(2)}_3(M)$. Hence, in this setting, we have $\chi(M)=b^{(2)}_2(M)$. As $b^{(2)}_2(M)=2h^{2,0}_{(2)}(M)+h^{1,1}_{(2)}(M)$ and $2h^{2,0}_{(2)}(M)-h^{1,1}_{(2)}(M)\neq 0$ we can conclude that $b^{(2)}_2(M)>0$ and therefore $\chi(M)>0$ as desired. Finally, according to Th. \ref{gamma2} and \eqref{please}, we have $$\chi(M,\mathcal{O}_M)=\sum_{q=0}^2(-1)^qh_{\overline{\partial}}^{0,q}(M)=\sum_{q=0}^2(-1)^qh_{(2)}^{0,q}(M,h)=\sum_{p=0}^2(-1)^ph_{(2)}^{p,0}(M,h)=h^{2,0}_{(2)}(M).$$ By Prop. \ref{tianjin} we know that $h^{2,0}_{(2)}(M)>0$. We can thus conclude that $\chi(M,\mathcal{O}_M)>0$.
\end{proof}

Using Prop. \ref{tianjin} and Prop. \ref{luccio} we have also the following applications.

\begin{cor}
\label{wuwu}
Let $(M,h)$ be a compact K\"ahler surface with infinite fundamental group and $\sigma(M)\neq 0$.  If $\chi(M)\leq 0$ then $2b_1^{(2)}(M)\geq b_2^{(2)}(M)>0$ and $h^{1,0}_{(2)}(M)>0$.
\end{cor}

\begin{proof}
Since the universal covering of $M$ has infinite volume  we have $b_0^{(2)}(M)=b_4^{(2)}(M)=0$. Hence for the Euler characteristic we have $\chi(M)=-b_1^{(2)}(M)+b_2^{(2)}(M)-b_3^{(2)}(M)=-2b_1^{(2)}(M)+b_2^{(2)}(M)$. Moreover we have $b^{(2)}_2(M)=2h^{2,0}_{(2)}(M)+h^{1,1}_{(2)}(M)$ and $2h^{2,0}_{(2)}(M)-h^{1,1}_{(2)}(M)\neq 0$ because we assumed that $\sigma(M)\neq 0$, see \eqref{sobsob}. Thus  we know that $b_2^{(2)}(M)>0$. Therefore, as we assumed that $\chi(M)\leq 0$, we can conclude that  $2b_1^{(2)}(M)\geq b_2^{(2)}(M) >0$.  Finally, since $b^{(2)}_1(M)=h^{1,0}_{(2)}(M)+h^{0,1}_{(2)}(M)$ and $h^{1,0}_{(2)}(M)=h^{0,1}_{(2)}(M)$, we can conclude that $h^{1,0}_{(2)}(M)>0$.
\end{proof}

\begin{cor}
Let $(M,h)$ be a compact K\"ahler surface with infinite fundamental group and $\sigma(M)\neq 0$.  Assume furthermore that $\chi(M)\leq 0$. Then the underlying smooth four dimensional manifold $M$ carries no Riemannian  metric with nonnegative Ricci curvature.
\end{cor}

\begin{proof}
By Cor. \ref{wuwu} we know that $b_1^{(2)}(M)>0$. Assume now that on $M$ there is a Riemannian metric $g$ such that $\ric(g)\geq 0$. First we point out that $g$ must vanish somewhere because $M$ has infinite fundamental group. Let $\pi:\tilde{M}\rightarrow M$ be the universal covering of $M$. Clearly  $\tilde{g}:=\pi^*g$ satisfies $\ric(\tilde{g})\geq 0$. Hence $(\tilde{M},\tilde{g})$ is a complete Riemannian manifold with nonnegative Ricci curvature and infinite volume. By \cite{Dod} Cor. 1 we know that $\mathcal{H}_2^{1}(\tilde{M},\tilde{g})=0$ which in turn implies $b_1^{(2)}(M)=0$. Since we have already observed that the assumptions of this statement imply $b_1^{(2)}(M)>0$ we can conclude that $M$ carries no Riemannian  metric with nonnegative Ricci curvature.
\end{proof}

We have now the following proposition which provides some sufficient conditions in order to show that a compact K\"ahler surface is projective algebraic.

\begin{prop}
Let $(M,h)$ be a compact K\"ahler surface with non-positive sectional curvatures. Assume that $\sigma(M)>0$. Then the Kodaira dimension of $M$ is $2$. In particular $M$ is  projective algebraic.
\end{prop}

\begin{proof}
Since $(M,h)$ has non-positive sectional curvatures we know  in particular that it is K\"ahler parabolic. Therefore, by Prop. \ref{tianjin}  and \ref{luccio} we know that $h^{2,0}_{(2)}(M)>0$, $\mathcal{H}^{2,0}_{2}(\tilde{M},\tilde{h})$ is infinite dimensional and $\chi(M,\mathcal{O}_M)>0$. Moreover by Wu's Theorem, see \cite{MaMa} pag. 283, we know that the universal covering of $M$, $\pi:\tilde{M}\rightarrow M$, is a Stein manifold. Therefore $\tilde{M}$ does not contain any compact complex submanifold of positive dimension. Summarizing $(M,h)$ is a compact K\"ahler surface with $\chi(M,\mathcal{O}_M)>0$ such that its universal covering $\pi:\tilde{M}\rightarrow M$ contains no compact complex submanifolds of positive dimension and $\mathcal{H}^{2,0}_{2}(\tilde{M},\tilde{h})$ is infinite dimensional. Now the conclusion follows by applying \cite{Gromov} pag. 287-288.
\end{proof}


Finally we have the last proposition of this section.

\begin{prop}
Let $(M,h)$ be a compact complex Hermitian surface with infinite fundamental group. Let $\mathcal{A}^{1,0}_M$ be the sheaf of holomorphic $(1,0)$-forms on $M$. Assume that $\chi(M)>0$ and that $\chi(M,\mathcal{A}^{1,0}_M)>0$. Let $\pi:\tilde{M}\rightarrow M$ be a non-compact Galois $\Gamma$-covering of $M$ and let  $\tilde{h}=\pi^*h$ be the pull-back metric on $\tilde{M}$.  Then $$h^{2,0}_{(2),\Gamma, \overline{\partial}}(M)>h^{2,1}_{(2),\Gamma, \overline{\partial}}(M).$$ In particular $h^{2,0}_{(2),\Gamma, \overline{\partial}}(M)>0$ and therefore the space of $L^2$-$(2,0)$-holomorphic forms on $(\tilde{M},\tilde{h})$ is infinite dimensional.
\end{prop}

\begin{proof}
According to Th. \ref{gammafrol} we know that $$\chi(M)=\sum_{q=0}^{2}(-1)^qh^{0,q}_{(2),\Gamma,\overline{\partial}}(M)-\sum_{q=0}^{2}(-1)^qh^{1,q}_{(2),\Gamma,\overline{\partial}}(M)+\sum_{q=0}^{2}(-1)^qh^{2,q}_{(2),\Gamma,\overline{\partial}}(M).$$
On the other hand, by Th. \ref{gamma2}, we know that $\chi(M,\mathcal{A}_M^{1,0})=\sum_{q=0}^{2}(-1)^qh^{1,q}_{(2),\Gamma,\overline{\partial}}(M)$. Hence, as we assumed $\chi(M)>0$ and $\chi(M,\mathcal{A}^{1,0}_M)>0$, we can deduce that $\sum_{q=0}^{2}(-1)^qh^{0,q}_{(2),\Gamma,\overline{\partial}}(M)+\sum_{q=0}^{2}(-1)^qh^{2,q}_{(2),\Gamma,\overline{\partial}}(M)>0$. Moreover applying \eqref{hhh} we get $\sum_{q=0}^{2}(-1)^qh^{0,q}_{(2),\Gamma,\overline{\partial}}(M)=\sum_{q=0}^{2}(-1)^qh^{2,q}_{(2),\Gamma,\overline{\partial}}(M)$ and this tells us  that $\sum_{q=0}^{2}(-1)^qh^{2,q}_{(2),\Gamma,\overline{\partial}}(M)>0$. Since $(\tilde{M},\tilde{h})$ has infinite volume we have $0=\mathcal{H}^{0,0}_2(\tilde{M},\tilde{h})\cong \mathcal{H}^{2,2}_2(\tilde{M},\tilde{h})$ which in turn implies $0=h^{0,0}_{(2),\Gamma, \overline{\partial}}(M)=h^{2,2}_{(2),\Gamma, \overline{\partial}}(M)$. In conclusion we showed that  $h^{2,0}_{(2),\Gamma,\overline{\partial}}(M)-h^{2,1}_{(2),\Gamma,\overline{\partial}}(M)>0$ and therefore we have also the inequality $h^{2,0}_{(2),\Gamma,\overline{\partial}}(M)>0$. Finally, applying  Lemma 15.10 in \cite{Roe}, we have that $\mathcal{H}^{2,0}_{2,\overline{\partial}}(\tilde{M},\tilde{h})$ is infinite dimensional. 
\end{proof}

\section{Further remarks on the $L^2$-Hodge numbers}
This section is concerned with the invariance of the $L^2$-Hodge numbers through bimeromorphic maps. We start by raising the following question: 
\begin{itemize}
\item Let $M_1$ and $M_2$ be two bimeromorphic compact complex manifolds of complex dimension $m$ with infinite fundamental groups. Is it true that $h^{p,0}_{(2)}(M_1)=h^{p,0}_{(2)}(M_2)$ and $h^{0,q}_{(2)}(M_1)=h^{0,q}_{(2)}(M_2)$ for each $p=0,...,m$ and $q=0,...,0$ respectively?
\end{itemize}
Below we collect some positive partial answers.

\begin{teo}
\label{nonso}
Let $M_1$ and $M_2$ be two compact complex manifolds of complex dimension $m$. Assume that there exists a modification $\phi:M_1\rightarrow M_2$.
Then we have the following equality: $$h^{m,0}_{(2),\overline{\partial}}(M_1)=h^{m,0}_{(2),\overline{\partial}}(M_2).$$
\end{teo}

In order to prove the above theorem we need the following propositions.

\begin{prop}
\label{group}
In the setting of Th. \ref{nonso}. Let $\pi_1(M_1)$ and $\pi_1(M_2)$ be the fundamentals group of $M_1$ and $M_2$ respectively. Then $\phi_*:\pi_1(M_1)\rightarrow \pi_1(M_2)$, the map that $\phi$ induces between  $\pi_1(M_1)$ and $\pi_1(M_2)$, is an isomorphism.
\end{prop}
\begin{proof}
This is a well known result. For the reader's convenience we give a proof. By the assumptions we know that there exists an analytic subset $W\subset M_2$ such that  $\phi|_{M_1\setminus Z}:M_1\setminus Z\rightarrow M_2\setminus W$ is a biholomorphism where $Z:=\phi^{-1}(W)$. Consider the inverse of $\phi$ as a map from $M_2$ to the power set of $M_1$. This is a meromorphic map, see \cite{ThP} pag. 289, whose set of points of indeterminacies is $W$. Since $M_1$ and $M_2$ are nonsingular they are in particular normal  and therefore, thanks to \cite{UK} Th. 2.5, we can conclude that $W$ has complex codimension greater or equal than 2. Since $Z$ and $W$ are  analytic subsets of $M_1$ and $M_2$ respectively we can decompose them as $Z=Q_1\cup Q_2\cup...\cup Q_l$, $W=N_1\cup N_2\cup...\cup N_n$ for some $l,n\in \mathbb{N}$ where $Q_a$ is a complex submanifold of $M_1$ and $N_k$ is a complex submanifold of $M_2$ for each $a=1,...,l$ and $k=1,...,n$ respectively. Moreover $Q_b\cap Q_d=\emptyset$ and  $N_r\cap N_s=\emptyset$ whenever  $b\neq d$ and $r\neq s$. Finally, thanks to the above remark, we have that the complex codimension of $N_k$ is greater or equal than 2 for each $k=0,...,n$. We recall now a classical application of Thom's transversality theorem: Let $P$ be a smooth manifold and let $L\subset P$ be a submanifold. Then the morphism $\pi_1(P\setminus L)\rightarrow \pi_1(P)$ induced by the inclusion $P\setminus L\hookrightarrow P$ is surjective if $\dim(P)-\dim(L)\geq 2$ and is an isomorphism if $\dim(P)-\dim(L)\geq 3$. Using this property it is not difficult to show that the inclusion $i:M_2\setminus (N_1\cup N_2\cup...\cup N_n)\rightarrow M_2$ induces an isomorphism $i_*:\pi_1(M_2\setminus (N_1\cup N_2\cup...\cup N_n))\rightarrow \pi_1(M_2)$. For the same reasons we can prove that the inclusion $j:M_1\setminus Z\rightarrow M_1$ induces a surjective morphism $j_*:\pi_1(M_1\setminus Z)\rightarrow \pi_1(M_1)$. Consider now the map $\phi|_{M_1\setminus Z}:M_1\setminus Z\rightarrow M_2\setminus W$. Clearly we have $i\circ \phi|_{M_1\setminus Z}:M_1\setminus Z\rightarrow M_2=\phi\circ j:M_1\setminus Z\rightarrow M_2$. On the other hand $(i\circ \phi|_{M_1\setminus Z})_*:\pi_1(M_1\setminus Z)\rightarrow \pi_1(M_2)$ is an isomorphism and $j_*:\pi_1(M_1\setminus Z)\rightarrow \pi_1(M_1)$ is a surjective morphism. Hence we can thus conclude  that $\phi_*:\pi_1(M_1)\rightarrow \pi_1(M_2)$ is an isomorphism as desired.
\end{proof}

\begin{prop}
\label{brick}
Let $M$ be a complex manifold of complex dimension $m$  and let $g$ and $h$ be two Hermitian metrics on $M$. Then we have an equality of Hilbert spaces $$L^2\Omega^{m,0}(M,g)=L^2\Omega^{m,0}(M,h).$$ 
\end{prop} 

\begin{proof}
The statement   follows  by the computations carried out in \cite{GMMI} pag. 145. 
\end{proof}

\begin{prop}
\label{monacchi}
Let $(M,g)$ be a compact complex Hermitian manifold and let $Z\subset M$ be an analytic subset. Then $(M\setminus Z,g|_{M\setminus Z})$ is parabolic. In other words there exists a sequence of functions ${\phi_n}\subset C^{\infty}_c(M\setminus Z)$ such that 
\begin{enumerate}
\item $0\leq \phi_n\leq 1$ for each $n\in \mathbb{N}$,
\item $\lim\phi_n=1$ pointwise as $n\rightarrow \infty$,
\item $\lim \|d\phi_n\|_{L^2\Omega^1(M\setminus Z,g|_{M\setminus Z})}=0$ as $n\rightarrow \infty$.
\end{enumerate}
\end{prop}

\begin{proof}
This follows arguing as in \cite{BeGu} or in \cite{JRU}.
\end{proof}

Let's go back now to the setting of Th. \ref{nonso}. As remarked in the proof of Prop. \ref{group} we know that there exists an analytic subset $W\subset M_2$ such that $\phi|_{M_1\setminus Z}:M_1\setminus Z\rightarrow M_2\setminus W$ is a biholomorphism where $Z:=\phi^{-1}(W)$. Let $\pi_1:\tilde{M}_1\rightarrow M_1$ be the universal covering of $M_1$, let $g_1$ be any Hermitian metric on $M_1$, let $\tilde{g}_1:=\pi_{1}^*g_1$,  let $X:=\pi_{1}^{-1}(Z)$ and finally let $A:=\tilde{M}_1\setminus X$. We have the following proposition:

\begin{prop}
\label{occaminonna}
In the setting described above, the Hodge-Dolbeault operator 
\begin{equation}
\label{pluf}
\overline{\partial}_{m}+\overline{\partial}_{m}^t:L^2\Omega^{m,\bullet}(A,\tilde{g}_1|_A)\rightarrow L^2\Omega^{m,\bullet}(A,\tilde{g}_1|_A)
\end{equation}
with  domain given by $\Omega^{m,\bullet}_c(A)$ is essentially self-adjoint. Moreover the unique closed extension of \eqref{pluf} coincides with the operator
\begin{equation}
\label{pluff}
\overline{\partial}_{m}+\overline{\partial}_{m}^t:L^2\Omega^{m,\bullet}(\tilde{M}_1,\tilde{g}_1)\rightarrow L^2\Omega^{m,\bullet}(\tilde{M}_1,\tilde{g}_1)
\end{equation}
where \eqref{pluff} is the unique closed extension of $\overline{\partial}_{m}+\overline{\partial}_{m}^t:\Omega^{m,\bullet}_c(\tilde{M}_1,\tilde{g}_1)\rightarrow \Omega^{m,\bullet}_c(\tilde{M}_1,\tilde{g}_1)$ viewed as an unbounded and densely defined operator acting on $L^2\Omega^{m,\bullet}(\tilde{M}_1,\tilde{g}_1)$.
\end{prop}
\begin{proof}
We adapt to our case \cite{FBei} Prop. 3.1 and so we will be brief.  As $\tilde{M}_1\setminus A$ has measure zero in $\tilde{M}_1$ we have an equality of Hilbert spaces $L^2\Omega^{m,\bullet}(\tilde{M}_1,\tilde{g}_1)=L^2\Omega^{m,\bullet}(A,\tilde{g}_1|_A)$. Let us label by $\mathcal{D}(\overline{\partial}_{m} + \overline{\partial}^t_{m})$, $\mathcal{D}((\overline{\partial}_{m} + \overline{\partial}^t_{m})_{\min})$ and $\mathcal{D}((\overline{\partial}_{m} + \overline{\partial}^t_{m})_{\max})$ respectively  the domain of \eqref{pluff}, the minimal domain of \eqref{pluf} and the maximal domain of \eqref{pluf}. 
As a first step we want to show that $\mathcal{D}(\overline{\partial}_{m} + \overline{\partial}^t_{m})=\mathcal{D}((\overline{\partial}_{m} + \overline{\partial}^t_{m})_{\min})$.
Since the inclusion $\mathcal{D}(\overline{\partial}_m + \overline{\partial}^t_m)\supset \mathcal{D}((\overline{\partial}_m + \overline{\partial}^t_m)_{\min})$
is clear we are left to prove the other inclusion $\mathcal{D}(\overline{\partial}_{m} + \overline{\partial}^t_{m})\subset
\mathcal{D}((\overline{\partial}_{m} + \overline{\partial}^t_{m})_{\min})$. As $(\tilde{M}_1,\tilde{g}_1)$ is complete    it is enough to prove that $\Omega^{m}_c(\tilde{M}_1)\subset \mathcal{D}((\overline{\partial}_{m} + \overline{\partial}^t_{m})_{\min})$.
According to Prop. \ref{monacchi} we know that $(M_1\setminus Z,g_1|_{M_1\setminus Z})$ is parabolic. Let $\{\phi_i\}\subset C_{c}^{\infty}(M_1\setminus Z)$ be a sequence of functions that makes $(M_1\setminus Z, g_1|_{M_1\setminus Z})$ parabolic.  Let $\{\tilde{\phi}_i\}\subset C^{\infty}(\tilde{M}_1)$ be defined as $\tilde{\phi}_i=\phi_i\circ \pi$.  Consider now any form $\omega\in \Omega^{m,\bullet }_c(\tilde{M}_1)$. Then $\{\tilde{\phi}_i\omega\}$ is a sequence of forms lying in $\Omega_c^{m,\bullet}(\tilde{M}_1)$ and, since $\omega\in L^{\infty}\Omega^{m,\bullet}(\tilde{M_1},\tilde{g}_1)$, we can argue  as in the proof of Prop. 3.1 in \cite{FBei} in order  to show that $\tilde{\phi}_i\omega\rightarrow \omega$ as $i\rightarrow \infty$ in the graph norm of \eqref{pluff}.
This tells us that   $\mathcal{D}(\overline{\partial}_{m} + \overline{\partial}^t_{m})\subset
\mathcal{D}((\overline{\partial}_{m} + \overline{\partial}^t_{m})_{\min})$ and thus $\mathcal{D}(\overline{\partial}_{m} + \overline{\partial}^t_{m})=
\mathcal{D}((\overline{\partial}_{m} + \overline{\partial}^t_{m})_{\min})$. Therefore the minimal extension of \eqref{pluf} coincides with \eqref{pluff}. Now, using the fact that $\overline{\partial}_{m} + \overline{\partial}^t_{m}:L^2\Omega^{m,\bullet}(\tilde{M}_1,\tilde{g}_1)\rightarrow L^2\Omega^{m,\bullet}(\tilde{M}_1,\tilde{g}_1)$
is self-adjoint, we   get that $((\overline{\partial}_{m} + \overline{\partial}^t_{m})_{\min})^*=(\overline{\partial}_{m} + \overline{\partial}^t_{m})_{\min}$. On the other hand we have $((\overline{\partial}_{m} + \overline{\partial}^t_{m})_{\min})^*=(\overline{\partial}_{m} + \overline{\partial}^t_{m})_{\max}$. Therefore we are lead to the conclusion that $(\overline{\partial}_{m} + \overline{\partial}^t_{m})_{\max}=(\overline{\partial}_{m} + \overline{\partial}^t_{m})_{\min}$  and this amounts  to saying that \eqref{pluf} is essentially self-adjoint.
\end{proof}

\begin{prop}
\label{lit}
Let $(\tilde{M},\tilde{g})$,  and $A$ be as in Prop. \ref{nonso}. Then the following three operators coincide:
\begin{align}
\label{ttt}
& \overline{\partial}_{m,q,\max}:L^2\Omega^{m,q}(A,\tilde{g}|_A)\rightarrow L^2\Omega^{m,q+1}(A,\tilde{g}|_A),\\
\label{bilo}
& \overline{\partial}_{m,q,\min}:L^2\Omega^{m,q}(A,\tilde{g}|_A)\rightarrow L^2\Omega^{m,q+1}(A,\tilde{g}|_A),\\
\label{ses}
& \overline{\partial}_{m,q}:L^2\Omega^{m,q}(\tilde{M},\tilde{g})\rightarrow L^2\Omega^{m,q+1}(\tilde{M},\tilde{g}),
\end{align}
where \eqref{ses} is the unique closed extension of  $\overline{\partial}_{m,q}:\Omega_c^{m,q}(\tilde{M})\rightarrow \Omega_c^{m,q+1}(\tilde{M})$ viewed as an unbounded and densely defined operator acting between $L^2\Omega^{m,q}(\tilde{M},\tilde{g})$ and  $L^2\Omega^{m,q+1}(\tilde{M},\tilde{g}).$
\end{prop}

\begin{proof}
This follows immediately by Prop. \ref{occaminonna} and Lemma 2.3 in \cite{BruLe}. 
\end{proof}

We have now all the ingredients to prove Th. \ref{nonso}.
\begin{proof}
In order to avoid any confusion with the notation, along the proof  we will label with $\Gamma_1$ and $\Gamma_2$ the fundamental groups of $M_1$ and $M_2$ respectively. We start by pointing out that $\pi_1|_A:A\rightarrow M_1\setminus Z$ is a Galois $\Gamma_1$-covering of $M_1\setminus Z$. Similarly, defining $Y:=\pi_2^{-1}(W)$ and $B:=\tilde{M_2}\setminus Y$, we have that $\pi_2|_B:B\rightarrow M_2\setminus W$ is a Galois $\Gamma_2$-covering of $M_2\setminus W$. Let $\tilde{\phi}$ be a lifting of $\phi$, that is a map $\tilde{\phi}:\tilde{M_1}\rightarrow \tilde{M_2}$ such that  $\pi_2\circ \tilde{\phi}=\phi\circ \pi_1$. Let $g_2$ be any Hermitian metric on $M_2$, let $\tilde{g_2}:=\pi_2^*g_2$ and let $\gamma:=\tilde{\phi}^*\tilde{g}_2$. Then $\gamma$ is a positive semidefinite Hermitian product on $\tilde{M_1}$ which is strictly positive on $A$. According to Prop. \ref{lit} we have $$\mathcal{H}^{m,0}_{2,\overline{\partial}_{\min}}(B,\tilde{g}_2|_{B})=\mathcal{H}^{m,0}_{2,\overline{\partial}_{\max}}(B,\tilde{g}_2|_{B})=\mathcal{H}^{m,0}_{2,\overline{\partial}}(\tilde{M}_2,\tilde{g}_2).$$ Moreover, as explained in the proof of Prop. \ref{group}, we know that the inclusions  $M_1\setminus Z\hookrightarrow M_1$ and $M_2\setminus W\hookrightarrow M_2$ induce 
a surjective map $\pi_1(M_1\setminus Z)\rightarrow \Gamma_1$ and an isomorphism $\pi_1(M_2\setminus Z)\cong \Gamma_2$ respectively. Furthermore we recall that $L^2\Omega^{m,0}(A,\gamma|_A)=L^2\Omega^{m,0}(A,\tilde{g}_1|_A)=L^2\Omega^{m,0}(\tilde{M}_1,\tilde{g}_1)$ as a consequence of Prop. \ref{brick} and the fact that $X$ has measure zero in $\tilde{M}_1$ with respect to $\dvol_{\tilde{g}_1}$. Altogether we can conclude that  $\mathcal{H}^{m,0}_{2,\overline{\partial}_{\max}}(A,\gamma|_A)$ and  $\mathcal{H}^{m,0}_{2,\overline{\partial}_{\max}}(B,\tilde{g}_2|_{B})$ have  the structure of  $\Gamma_1$-module and  $\Gamma_2$-module of $L^2\Omega^{m,0}(\tilde{M}_1,\tilde{g}_1)$ and $L^2\Omega^{m,0}(\tilde{M}_2,\tilde{g}_2)$ respectively. We also know that $\phi_*:\Gamma_1\rightarrow \Gamma_2$ is an isomorphism and moreover it is clear that $\tilde{\phi}$ induces an isomorphism between the $\Gamma_2$-module $\mathcal{H}^{m,0}_{2,\overline{\partial}_{\max}}(B,\tilde{g}_2|_{B})$ and the $\Gamma_1$-module $\mathcal{H}^{m,0}_{2,\overline{\partial}_{\max}}(A,\gamma|_A)$. Thus, in order to conclude the proof, it suffices to show that 
\begin{equation}
\label{potager}
\mathcal{H}^{m,0}_{2,\overline{\partial}_{\max}}(A,\gamma|_A)=\mathcal{H}^{m,0}_{2,\overline{\partial}}(\tilde{M}_1,\tilde{g}_1).
\end{equation}
Let $\omega\in \mathcal{H}^{m,0}_{2,\overline{\partial}_{\max}}(A,\gamma|_A)$. By definition 
\begin{equation}
\label{pappa}
\omega\in \mathcal{D}(\overline{\partial}_{m,0,\max})\subset L^2\Omega^{m,0}(A,\gamma|_A)\ \text{and}\  \overline{\partial}_{m,0,\max}\omega=0.
\end{equation}
This in turn implies that 
\begin{equation}
\label{maxmax}
(\overline{\partial}_{m,0,\max})^*( \overline{\partial}_{m,0,\max} \omega)=0
\end{equation}
where $(\overline{\partial}_{m,0,\max})^*:L^2\Omega^{m,1}(A,\gamma|_A)\rightarrow L^2\Omega^{m,0}(A,\gamma|_A)$ is the Hilbert space adjoint of $\overline{\partial}_{m,0,\max}:L^2\Omega^{m,0}(A,\gamma|_A)\rightarrow L^2\Omega^{m,1}(A,\gamma|_A)$. Therefore, thanks to  \eqref{maxmax}, we know that  $\omega$
is in the null space of the maximal extension of $\Delta_{\overline{\partial},m,0}:L^2\Omega^{m,0} (A, \gamma|_A)\rightarrow L^2\Omega^{m,0} (A, \gamma|_A)$. As this latter operator is elliptic we can conclude that $\omega$  is smooth on $A$ and thus, by \eqref{pappa}, that $\omega$ is holomorphic on $A$. This, together with Prop. \ref{brick}, tells us that $\omega\in \mathcal{H}^{m,0}_{2,\overline{\partial}_{\max}}(A,\tilde{g}_1|_A)$ and finally, using Prop. \ref{lit}, we can conclude that $\omega\in \mathcal{H}^{m,0}_{2,\overline{\partial}}(\tilde{M}_1,\tilde{g}_1)$. So we proved that  $\mathcal{H}^{m,0}_{2,\overline{\partial}_{\max}}(A,\gamma|_A)\subset \mathcal{H}^{m,0}_{2,\overline{\partial}}(\tilde{M}_1,\tilde{g}_1)$. The reversed inclusion is straightforward. Indeed let $\omega\in \mathcal{H}^{m,0}_{2,\overline{\partial}}(\tilde{M}_1,\tilde{g}_1)$. Clearly $\omega$ is a holomorphic $(m,0)$-form on $\tilde{M_1}$ that lies in $L^2\Omega^{m,0}(\tilde{M}_1,\tilde{g}_1)$. Therefore $\omega$ is  holomorphic  on $A$ and it lies in $L^2\Omega^{m,0}(A,\tilde{g}_1|_A)$. In other words $\omega$ is  holomorphic  on $A$ and, thanks to Prop. \ref{brick}, it lies in $L^2\Omega^{m,0}(A,\gamma|_A)$, that is $\omega\in \mathcal{H}^{m,0}_{2,\overline{\partial}_{\max}}(A,\gamma|_A)$ as desired. The proof is thus concluded.
\end{proof}

We have the following applications of Th. \ref{nonso}.

\begin{cor}
\label{daidai}
Let $M_1$ and $M_2$ be two compact complex manifolds of complex dimension $m$. Assume that they are bimeromorphic. Then $$h^{m,0}_{(2),\overline{\partial}}(M_1)=h^{m,0}_{(2),\overline{\partial}}(M_2).$$
\end{cor}
\begin{proof}
This follow immediately by Th. \ref{nonso} and the definition of bimeromorphic map.
\end{proof}

We conclude the paper with the following proposition.

\begin{prop}
Let $M_1$ and $M_2$ be two compact K\"ahler  surfaces with infinite fundamental groups. Assume that they are bimeromorphic. Then we have the  equalities:
\begin{equation}
\label{chedire}
h^{p,0}_{(2)}(M_1)=h^{p,0}_{(2)}(M_2),\ p=0,1,2\\
\quad\quad\quad h^{0,q}_{(2)}(M_1)=h^{0,q}_{(2)}(M_2),\ q=0,1,2.
\end{equation}
Moreover the following two properties are equivalent:
\begin{enumerate}
\item $h^{1,1}_{(2)}(M_1)=h^{1,1}_{(2)}(M_2).$
\item $h^{1,1}(M_1)=h^{1,1}(M_2).$
\end{enumerate}
\end{prop}
\begin{proof}
 Since $M_1$ and $M_2$ are bimeromorphic it is known that $h^{0,q}(M_1)=h^{0,q}(M_2)$ with $q=0,1,2$ and analogously   $h^{p,0}(M_1)=h^{p,0}(M_2)$ with $p=0,1,2$. By duality this tells us that $h^{2,q}(M_1)=h^{2,q}(M_2)$ with $q=0,1,2$ and analogously   $h^{p,2}(M_1)=h^{p,2}(M_2)$ with $p=0,1,2$. Thanks to Cor.\ref{daidai} we know that $h^{2,0}_{(2)}(M_1)=h^{2,0}_{(2)}(M_2)$ and therefore by duality $h^{0,2}_{(2)}(M_1)=h^{0,2}_{(2)}(M_2)$.  Moreover $h^{0,0}_{(2)}(M_1)=h^{0,0}_{(2)}(M_2)=0$ and, since we know that $\chi(M_1,\mathcal{O}_{M_1})=h^{0,2}_{(2)}(M_1)-h^{0,1}_{(2)}(M_1)$ that  $\chi(M_2,\mathcal{O}_{M_2})=h^{0,2}_{(2)}(M_2)-h^{0,1}_{(2)}(M_2)$ and that $\chi(M_1,\mathcal{O}_{M_1})=\chi(M_2,\mathcal{O}_{M_2})$, we can conclude that $h^{0,1}_{(2)}(M_1)=h^{0,1}_{(2)}(M_2)$. Again by duality we have $h^{2,1}_{(2)}(M_1)=h^{2,1}_{(2)}(M_2)$, $h^{1,0}_{(2)}(M_1)=h^{1,0}_{(2)}(M_2)$ and $h^{1,2}_{(2)}(M_1)=h^{1,2}_{(2)}(M_2)$. This shows \eqref{chedire}. Assume now that $h^{1,1}_{(2)}(M_1)=h^{1,1}_{(2)}(M_2)$. Then we have $h^{p,q}_{(2)}(M_1)=h^{p,q}_{(2)}(M_2)$ for each $p=0,1,2$ and $q=0,1,2$. Consequently by Th. \ref{gammahodge} we have $\sigma(M_1)=\sigma(M_2)$. On the other hand we already know that $h^{p,q}(M_1)=h^{p,q}(M_2)$ whenever $(p,q)\neq (1,1)$. Hence, as a consequence of the classical Hodge index theorem, we have $h^{1,1}(M_1)=h^{1,1}(M_2)$ as desired. Conversely assume that $h^{1,1}(M_1)=h^{1,1}(M_2)$. Then $\sigma(M_1)=\sigma(M_2)$. Therefore, as we already know that $h^{p,q}_{(2)}(M_1)=h^{p,q}_{(2)}(M_2)$ whenever $(p,q)\neq (1,1)$, we can conclude, thanks to Th.\ref{gammahodge}, that $h^{1,1}_{(2)}(M_1)=h^{1,1}_{(2)}(M_2)$. The proof is thus complete. 
\end{proof}

\end{document}